\documentclass[smallextended]{svjour3}
\usepackage{graphicx}
\usepackage{float}
\usepackage{mathtools}
\usepackage{wrapfig}
\usepackage{amssymb, amsmath, latexsym}
\usepackage{nicefrac}
\usepackage{hhline}
\usepackage{multirow}
\usepackage{pifont}
\usepackage[colorinlistoftodos,bordercolor=orange,backgroundcolor=orange!20,linecolor=orange,textsize=scriptsize]{todonotes}
\usepackage{algorithm}\usepackage{algpseudocode}
\usepackage[symbol]{footmisc}

\usepackage{array, makecell} %
\usepackage{cite}
\newtheorem{assumption}{Assumption}

\def \R {\mathbb R}

\newcommand{\EndProof}{\begin{flushright}$\square$\end{flushright}}

\def\R{\mathbb{R}}
\newcommand{\E}{{\mathbb E}}

\def\R{\mathbb R}
\def\E{\mathbb E}

\def\la{\langle}
\def\ra{\rangle}

\begin{document}

\title{Improved Exploiting Higher Order Smoothness in Derivative-free Optimization and Continuous Bandit\thanks{This work is based on results achieved by 63 Conference MIPT held in November 2020.  The research of A. Gasnikov was partially supported by the Ministry of Science and Higher Education of the Russian Federation (Goszadaniye) 075-00337-20-03, project no. 0714-2020-0005. The work of V. Novitskii was supported by Andrei M. Raigorodskii Scholarship in Optimization.}}

\titlerunning{Improved  Exploiting Higher Order Smoothness in Derivative-free Optimization}

\author{Vasilii Novitskii \and
        Alexander Gasnikov
}
\institute{V. Novitskii \at
            Moscow Institute of Physics and Technology, Russia\\
            \email{vasiliy.novitskiy@phystech.edu} 
            \and
            A. Gasnikov \at
            Moscow Institute of Physics and Technology, Russia \\
            Institute for Information Transmission Problems RAS, Russia \\
            Weierstrass Institute for Applied Analysis and Stochastics, Germany
}
\authorrunning{V. Novitskii  and A. Gasnikov}
%

\date{Received: date / Accepted: date}
\maketitle              
\begin{abstract}
We consider $\beta$-smooth (satisfies the generalized Hölder condition with parameter $\beta > 2$) 
stochastic convex optimization problem with zero-order one-point oracle. The best known result was \cite{akhavan2020exploiting}:  $$\mathbb{E} \left[f(\overline{x}_N) - f(x^*)\right] = \tilde{\mathcal{O}} \left(\dfrac{n^{2}}{\gamma N^{\frac{\beta-1}{\beta}}} \right)$$ in $\gamma$-strongly convex case, where $n$ is the dimension. In this paper we improve this bound:
$$\mathbb{E} \left[f(\overline{x}_N) - f(x^*)\right] = \tilde{\mathcal{O}} \left(\dfrac{n^{2-\textcolor{red}{\frac{1}{\beta}}}}{\gamma N^{\frac{\beta-1}{\beta}}} \right).$$


\keywords{zeroth-order optimization \and convex problem \and stochastic optimization \and one-point bandit \and smoothing kernel}

\end{abstract}

\clearpage

\section{Introduction}\label{section:intro}

We study the problem of zero-order stochastic optimization in which the aim is to minimize an unknown convex or strongly convex function where no gradient realization is given but a function value is available at each iteration with some additive noise $\xi$.
We also study a closely related problem of continuous stochastic bandits. These problems have received significant attention in the literature (see \cite{akhavan2020exploiting, bach2016highly, gasnikov2014stochastic, gasnikov2015gradient, gasnikov2017stochastic, zhang2020boosting, bubeck2017kernel, larson2019derivative-free, conn2009introduction, spall2003introduction}) and are fundamental for many application where the derivative of function is not available or it is hard to calculate derivatives.

The goal of this paper is to exploit higher order smoothness of the function to improve the performance of projected gradient-like algorithms. 
Our approach is outlined in Algorithm \ref{algo1}, in which a sequential algorithm gets at each iteration two function values under some noise. At each iteration the algorithm gets function values at points $x_k + \delta_k$ and $x_k - \delta_k$, where $\delta_k = \tau_k r_k e_k$. Here $r_k$ is uniformly distributed random variable, $e_k$ is uniformly distributed on the Euclidean sphere, $\tau_k$ -- is tunable parameter of the algorithm, the smaller $\tau_k$ is, the smaller approximation error of the gradient $\|\widetilde{g_k} - \nabla f(x_k)\|$ is (in this article we use only the Euclidean norm) but the bigger variance of $\|\widetilde{g_k}\|$ is, so the trade-off between these terms is needed. Our approach uses kernel smoothing technique proposed by Polyak and Tsybakov in \cite{polyak1990optimal}, this helps to exploit higher order smoothness.

\begin{algorithm}[H]
\caption{Zero-order Stochastic Projected Gradient} \label{algo1}
\begin{algorithmic}
\State 
\noindent {\bf Requires: } Kernel $K: [-1, 1] \rightarrow \mathbb{R}$, step size $\alpha_k>0$, parameters $\tau_k$.

\State
{\bf Initialization: } Generate scalars $r_1, \dots, r_N$ uniformly on $[-1,1]$ and vectors $e_1, \dots, e_N$ uniformly on the Euclidean unit sphere $S_n=\{e\in \mathbb{R}^n: \, \|e\|=1 \}$.
\For{$k=1, \dots, N$}{
 \begin{enumerate}
     \item $y_k := f(x_k+\tau_k r_k e_k) + \xi_k$, $y_k' := f(x_k-\tau_k r_k e_k) + \xi_k'$
     \item Define $\widetilde{g_k}:= \frac{n}{2\tau_k}(y_k-y_k')e_kK(r_k)$
     \item Update $x_{k+1} := \Pi_{Q}(x_k-\alpha_k\widetilde{g_k})$
 \end{enumerate}
}
\EndFor
\State 
\noindent {\bf Output:} $\left\{x_k\right\}_{k=1}^N$.
\end{algorithmic}
\end{algorithm}

In algorithms like Algorithm \ref{algo1} the two possibilities are usually considered. The first one is to obtain a function value in one point with some noise ("one-point" multi-armed bandit), the second is to observe function values in two points with the same noise at each iteration ("two-point" multi-armed bandit). The use of three and more points do not make dramatic difference to the results for two points \cite{duchi2015optimal}.
Note that despite our algorithm gets two function values for iteration, they are obtained with different noise $\xi_k$ and $\xi'_k$, so it is correct to regard Algorithm \ref{algo1} one-point and to compare it with one-point algorithms.  

In this paper we study higher order smooth functions $f$ functions satisfying the generalized Hölder condition with parameter $\beta > 2$ (see inequality \eqref{Hölder-condition} below). 

We address the question: what is the performance of Algorithm \ref{algo1}, namely the explicit dependency of the convergence rate on the main parameters $n$ (dimension), $N$, $\gamma$ (strong convexity parameter for strongly convex functions), $\beta$.
To handle this task we prove upper bound for Algorithm \ref{algo1}.

{\bf Contributions.} Out main contributions can be summarized as follows:
\begin{enumerate}
    \item For strongly-convex case: under an adversarial noise assumption (see Assumption \ref{noise-ass}) we establish for all $\beta > 2$ the upper bound of order $ \mathcal{O} \left(\dfrac{n^{2-\textcolor{red}{\frac{1}{\beta}}}}{\gamma N^{\frac{\beta-1}{\beta}}} \right)$ for the optimization error of Algorithm \ref{algo1} for strongly convex case.
    \item For convex case: under an adversarial noise assumption (see Assumption \ref{noise-ass}) we establish for all $\beta > 2$ that after $N(\varepsilon) = \mathcal{O} \left(\dfrac{n^{2+\textcolor{red}{\frac{1}{\beta-1}}}}{\varepsilon^{2+\frac{2}{\beta-1}}} \right)$  iterations of Algorithm \ref{algo1} for the regularized function $f_{\gamma}(x):=f(x)+\frac{\varepsilon}{2R^2} \| x-x_0 \|^2$ we achieve the optimization error less than or equal to $\varepsilon$.
\end{enumerate}

For clarity we compare our results with state-of-the-art ones in Table \ref{tab-1} (dependence of optimization error $\varepsilon$ on the number of iteration $N$, dimension $n$ and $\beta$, $\gamma$) and Table \ref{tab-2} (dependence of the number of iteration $N$ on the optimization error $\varepsilon$, dimension $n$ and $\beta$, $\gamma$). To summarize the results we use $\tilde{\mathcal{O}}()$ , where $\tilde{\mathcal{O}}()$ coincides with $\mathcal{O}()$ up to the logarithmic factor. 

\begin{table}[H]\caption{The dependence of optimization error ($\varepsilon$) on $N$ (number of iterations), $n$ (dimension), $\gamma$, $\beta$}\label{tab-1}
\centering
\begin{tabular}{|c|c|c|}
\hline
& strongly convex & convex\\ \hline
lower bound \cite{akhavan2020exploiting}& $\mathcal{O} \left( \min \left(\dfrac{n}{\gamma N^{\frac{\beta-1}{\beta}}} , \dfrac{n}{\sqrt{N}}\right) \right)$ & $\textcolor{blue}{\mathcal{O} \left(\min \left(\dfrac{\sqrt{n}}{N^{\frac{\beta-1}{2\beta}}}, \dfrac{n}{\sqrt{N}} \right) \right)}$\\ \hline
\makecell{this work\\ (2020)} & $\tilde{\mathcal{O}} \left(\dfrac{n^{2-\textcolor{red}{\frac{1}{\beta}}}}{\gamma N^{\frac{\beta-1}{\beta}}} \right)$ & $\tilde{\mathcal{O}} \left(\dfrac{n^{1-\textcolor{red}{\frac{1}{2\beta}}}}{ N^{\frac{\beta-1}{2\beta}}} \right)$\\ \hline
\makecell{Akhavan, Pontil, \\ Tsybakov (2020) \cite{akhavan2020exploiting}} & $\tilde{\mathcal{O}} \left(\dfrac{n^{2}}{\gamma N^{\frac{\beta-1}{\beta}}} \right)$ & $\textcolor{blue}{\tilde{\mathcal{O}} \left(\dfrac{n}{N^{\frac{\beta-1}{2\beta}}} \right)}$\\ \hline
\makecell{Bach, Perchet \\(2016) \cite{bach2016highly}} & $\mathcal{O} \left(\dfrac{n^{2-\frac{2}{\beta+1}}}{(\gamma N)^{\frac{\beta-1}{\beta+1}}} \right)$  & $\mathcal{O} \left(\dfrac{n^{1-\frac{1}{\beta+1}}}{N^{\frac{\beta-1}{2(\beta+1)}}} \right)$\\ \hline
\makecell{Gasnikov and al. \\(2015), $\beta=2$, \cite{gasnikov2017stochastic}} & $\tilde{\mathcal{O}} \left(\dfrac{n}{\sqrt{\gamma N}} \right)$ & $\tilde{\mathcal{O}} \left(\dfrac{\sqrt{n}}{N^{\nicefrac 1 4}} \right)$\\ \hline
\makecell{Akhavan, Pontil, \\ Tsybakov (2020), \\ special case $\beta=2$ \cite{akhavan2020exploiting}} & $\tilde{\mathcal{O}} \left(\dfrac{n}{\sqrt{\gamma N}} \right)$ & $\textcolor{blue}{\tilde{\mathcal{O}} \left(\dfrac{\sqrt{n}}{N^{\nicefrac 1 4}} \right)}$\\ \hline
\makecell{Zhang and al. \\ (2020) \cite{zhang2020boosting}} & $\textcolor{blue}{\mathcal{O} \left(\dfrac{n}{\sqrt{\gamma N}} \right)}$ & $\mathcal{O} \left(\dfrac{\sqrt{n}}{N^{\nicefrac 1 4}} \right)$\\ \hline
\end{tabular}
\end{table}

\begin{table}[H]\caption{The dependence of $N$ (number of iterations) on $\varepsilon$, $n$ (dimension), $\gamma$, $\beta$}\label{tab-2}
\centering
\begin{tabular}{|c|c|c|}
\hline
 & \makecell{strongly\\ convex} & convex \\ \hline
 lower bound \cite{akhavan2020exploiting}& $\mathcal{O} \left(\min\left(\dfrac{n^{1+\frac{1}{\beta-1}}}{(\gamma\varepsilon)^{\frac{\beta}{\beta-1}}}, \dfrac{n^2}{\varepsilon^2}\right) \right)$ & $\textcolor{blue}{\mathcal{O} \left(\min \left(\dfrac{n^{1+\frac{1}{\beta-1}}}{\varepsilon^{2+\frac{2}{\beta-1}}}, \dfrac{n^2}{\varepsilon^2}\right) \right)}$ \\ \hline
 \makecell{this work\\ (2020)} & $\tilde{\mathcal{O}} \left(\dfrac{n^{2+\textcolor{red}{\frac{1}{\beta-1}}}}{(\gamma\varepsilon)^{\frac{\beta}{\beta-1}}} \right)$ & $\tilde{\mathcal{O}} \left(\dfrac{n^{2+\textcolor{red}{\frac{1}{\beta-1}}}}{\varepsilon^{2+\frac{2}{\beta-1}}} \right)$ \\ \hline
 \makecell{Akhavan, Pontil, \\ Tsybakov (2020) \cite{akhavan2020exploiting}} & $\tilde{\mathcal{O}} \left(\dfrac{n^{2+\frac{2}{\beta-1}}}{(\gamma\varepsilon)^{\frac{\beta}{\beta-1}}} \right)$ & $\textcolor{blue}{\tilde{\mathcal{O}} \left(\dfrac{n^{2+\frac{2}{\beta-1}}}{\varepsilon^{2+\frac{2}{\beta-1}}} \right)}$ \\ \hline
 \makecell{Bach, Perchet \\(2016) \cite{bach2016highly}}& $\mathcal{O} \left(\dfrac{n^{2+\frac{2}{\beta-1}}}{\gamma\varepsilon^{\frac{\beta+1}{\beta-1}}} \right)$ & $\mathcal{O} \left(\dfrac{n^{2+\frac{2}{\beta-1}}}{\varepsilon^{2+\frac{2}{\beta-1}}} \right)$ \\ \hline
 \makecell{Gasnikov and al. \\(2015), $\beta=2$ \cite{gasnikov2017stochastic}} & $\tilde{\mathcal{O}} \left(\dfrac{n^{2}}{\gamma\varepsilon^2} \right)$ & $\tilde{\mathcal{O}} \left(\dfrac{n^{2}}{\varepsilon^{3}} \right)$ \\ \hline
  \makecell{Akhavan, Pontil, \\ Tsybakov (2020), \\ special case $\beta=2$ \cite{akhavan2020exploiting}} & $\tilde{\mathcal{O}} \left(\dfrac{n^{2}}{\gamma\varepsilon^2} \right)$ & $\textcolor{blue}{\tilde{\mathcal{O}} \left(\dfrac{n^{2}}{\varepsilon^{3}} \right)}$ \\ \hline
  \makecell{Zhang and al. \\ (2020) \cite{zhang2020boosting}} & $\textcolor{blue}{\mathcal{O} \left(\dfrac{n^{2}}{\gamma\varepsilon^2} \right)}$ & $\mathcal{O} \left(\dfrac{n^{2}}{\varepsilon^{3}} \right)$ \\ \hline
\end{tabular}
\end{table}

{\bf Comments on Table \ref{tab-1} and Table \ref{tab-2}}. 
\begin{enumerate}
    \item Note that in Table \ref{tab-1} and Table \ref{tab-2}  the right column equals to the central one by $\gamma \sim \varepsilon$.
    \item Note that the results of this work have better dependency  $\varepsilon(N)$ or $N(\varepsilon)$ than Gasnikov's one-point method only if $\beta > 2$ else another technique in Theorem \ref{theorem_1} is better (see \cite{gasnikov2017stochastic} or Theorem 5.1 in \cite{akhavan2020exploiting}). The result in this work is achieved using both kernel smoothing technique and measure concentration inequalities.
    \item The lower bound for strongly convex case is got under conditions $\gamma \geq N^{-\nicefrac{1}{2} + \nicefrac{1}{\beta}}$ (otherwise it is better to use convex methods) and (see \cite{akhavan2020exploiting}) $2\gamma \leq \max\limits_{x\in Q} \| \nabla f(x) \|$ . 
    \item The bounds marked in blue are not given in this article and in references but they can be got. 
    \item Too optimistic bounds $\mathcal{O} \left(\dfrac{n^{2-\frac{4}{\beta+1}}}{(\gamma N)^{\frac{\beta-1}{\beta+1}}} \right)$ and $\mathcal{O} \left(\dfrac{n^{2}}{\varepsilon^{2+\frac{2}{\beta-1}}} \right)$ were claimed in \cite{bach2016highly} instead of $\mathcal{O} \left(\dfrac{n^{2-\frac{2}{\beta+1}}}{(\gamma N)^{\frac{\beta-1}{\beta+1}}} \right)$ and $\mathcal{O} \left(\dfrac{n^{2+\frac{2}{\beta-1}}}{\varepsilon^{2+\frac{2}{\beta-1}}} \right)$, but Akhavan, Pontil and Tsybakov \cite{akhavan2020exploiting} found error in Lemma 2 in \cite{bach2016highly} where factor $d$ of dimension ($n$ in our notation) is missing.
\end{enumerate}

\section{Preliminaries}\label{section:preliminaries}
In this section we give the necessary notation, definitions and assumptions.

\subsection{Notation}
Let $\la \cdot, \cdot \ra$ and $\| \cdot \|$ be the standard inner product and Euclidean norm on $\R^n$ respectively. For every closed convex set $Q \subset \R^n$ and for every $x\in \R^n$ let $\Pi_Q(x)$ denote the Euclidean projection of $x$ on $Q$. 

\subsection{Problem}
We address the conditional minimization problem
\begin{equation*}
    f(x) \rightarrow \min_{x \in Q},
\end{equation*}
where $f: U_{\varepsilon_0}(Q) \rightarrow \mathbb{R}$ -- function (convex or strongly convex), 
$Q \subset \mathbb{R}^n$ -- convex compact set (Euclidean metrics). 

The optimization problem can be formulated as follows: find the sequence $\{ x_k \}_{k=1}^{N} \subset Q$ minimizing the average regret:
\begin{equation*}
    \dfrac{1}{N}\sum\limits_{k=1}^N \mathbb{E} \left[f(x_k) - f(x^*)\right].
\end{equation*} 
If the average regret is less than or equal to $\varepsilon$ then the optimization error of averaged estimator $\overline{x}_N = \frac{1}{N}\sum\limits_{k=1}^N x_k$ is also less than or equal to $\varepsilon$:
\begin{equation*}
    \mathbb{E} \left[f(\overline{x}_N) - f(x^*)\right] \leq \dfrac{1}{N}\sum\limits_{k=1}^N \mathbb{E} \left[f(x_k) - f(x^*)\right] \leq \varepsilon.
\end{equation*} 

\subsection{Noise}
The function values $f(x_k+\tau_k r_k e_k)$ and $f(x_k-\tau_k r_k e_k) $ are given with additive noise $\xi_k$ and $\xi'_k$ respectively (see Algorithm \ref{algo1}). Recall that the Algorithm \ref{algo1} is randomized: the scalars $r_1, \dots, r_N$ are distributed uniformly on $[-1,1]$ and the  vectors $e_1, \dots, e_N$ are distributed uniformly on the Euclidean unit sphere $S_n=\{e\in \mathbb{R}^n: \, \|e\|=1 \}$. 

\begin{assumption}\label{noise-ass}
    For all $k = 1, 2, \dots, N$ it holds that
    \begin{enumerate}
        \item\label{i} $\E[\xi_k^2] \leq \sigma^2$ and $\E[{\xi'}_k^2] \leq \sigma^2$ where $\sigma \geq 0$;
        \item\label{ii} the random variables $\xi_k$ and $\xi'_k$ are independent from $e_k$ and $r_k$, the random variables $e_k$ and $r_k$ are independent.
    \end{enumerate}
\end{assumption} 

We do not assume here neither zero-mean of $\xi_k$ and $\xi'_k$ nor i.i.d of $\{\xi_k\}_{k=1}^{N}$ and $\{\xi'_k\}_{k=1}^{N}$ as condtition \ref{ii} from assumption \ref{noise-ass} allows to avoid that.

\subsection{Higher order smoothness}

Let $l$ denote maximal integer number strictly less than $\beta$. Let ${\cal F}_{\beta}(L)$ denote the set of all functions $f: \mathbb{R}^n \rightarrow \mathbb{R}$ which are differentiable $l$ times and for all $x,z\in U_{\varepsilon_0} (Q)$ satisfy Hölder condition:
\begin{equation}\label{Hölder-condition}
    \Biggl| f(z) - \sum_{0\leq|m|\leq l} \dfrac{1}{m!} D^m f(x) (z-x)^m \Biggr| \leq L\|z-x\|^{\beta},
\end{equation}
where $L>0$, the sum is over multi-index $m=(m_1, \dots, m_n) \in \mathbb{N}^n$, we use the notation $m! = m_1! \cdot \dots \cdot m_n!$, $|m| = m_1 + \dots + m_n$ and we defined
\begin{equation*}
     D^m f(x) z^m = \dfrac{\partial^{|m|} f(x)}{\partial^{m_1} x_1  \dots \partial^{m^n} x_n} z_1^{m_1} \cdot \dots \cdot z_n^{m_n}, \; \forall z=(z_1, \dots, z_n) \in \R^n.
\end{equation*}

Let ${\cal F}_{\gamma, \beta}(L)$ denote the set of $\gamma$-strongly convex functions $f \in {\cal F}_{\beta}(L)$. Recall that $f$ is called $\gamma$-strongly convex for some $\gamma > 0$ if for all $x, z \in \R^n$ it holds that $f(z) \geq f(x) + \langle \nabla f(x), z-x \rangle + \frac{\gamma}{2} \|x-z\|^2$.

\subsection{Kernel}
For gradient estimator $\widetilde{g_k}$ we use the kernel 
\begin{equation*}
    K: [-1, 1] \rightarrow \mathbb{R},
\end{equation*}
satisfying
\begin{equation}\label{kernel-properties}
    \E[K(r)] = 0, \,
    \E[rK(r)] = 1, \, 
    \E [r^j K(r)] = 0,\, j=2, \dots, l, \,
    \E\left[ |r|^{\beta}|K(r)|\right] \leq \infty,
\end{equation}
where $r$ is a uniformly distributed on $[-1, 1]$ random variable. This helps us to get better bounds on the gradient bias $\| \widetilde{g_k} - \nabla f(x_k) \|$ (see Theorem \ref{theorem_1} for details). 

A weighted sum of Legendre polynoms is an example of such kernels:
\begin{equation}\label{kernels}
    K_{\beta}(r) := \sum\limits_{m=0}^{l(\beta)} p_{m}'(0)p_m(r),
\end{equation}
where $l(\beta)$ is maximal integer number strictly less than $\beta$ and $p_m(r) = \sqrt{2m+1} L_m(r)$, $L_m(u)$ is Legendre polynom. We have
\begin{equation*}
    \mathbb E \left[ p_m p_{m'}\right]=\delta(m-m').
\end{equation*}

As $\{p_m(r)\}_{m=0}^{j}$ is a basis for polynoms of degree less than or equal to $j$ we can represent $u^j := \sum\limits_{m=0}^{j} b_m p_m(r)$ for some integers $\{b_m\}_{m=0}^{j}$ (they depend on $j$). 

Let's calculate the expectation

\begin{equation*}
    \mathbb E \left[ r^j K_{\beta}(r)\right] = \sum\limits_{m=0}^{j} b_m p_m'(0) = (r^j)'|_{r=0} = \delta(j-1),
\end{equation*}
here $\delta(0) = 1$ and $\delta(x) = 1$ if $x \neq 0$. We proved that the presented $K_{\beta}(r)$ satisfies \eqref{kernel-properties}. We have the following kernels for different betas (see Figure \ref{fig:kernels}):

\begin{eqnarray*}
    &K_{\beta}(r) = 3r, \quad &\beta \in [2, 3],\\
    &K_{\beta}(r) = \dfrac{15r}{4} (5-7r^2), \quad &\beta \in (3, 5],\\
    &K_{\beta}(r) = \dfrac{105r}{64} (99r^4 - 126r^2 + 35), \quad &\beta \in (5, 7].
\end{eqnarray*}

\begin{figure}[h]
\centering
\includegraphics[width =  0.78\linewidth]{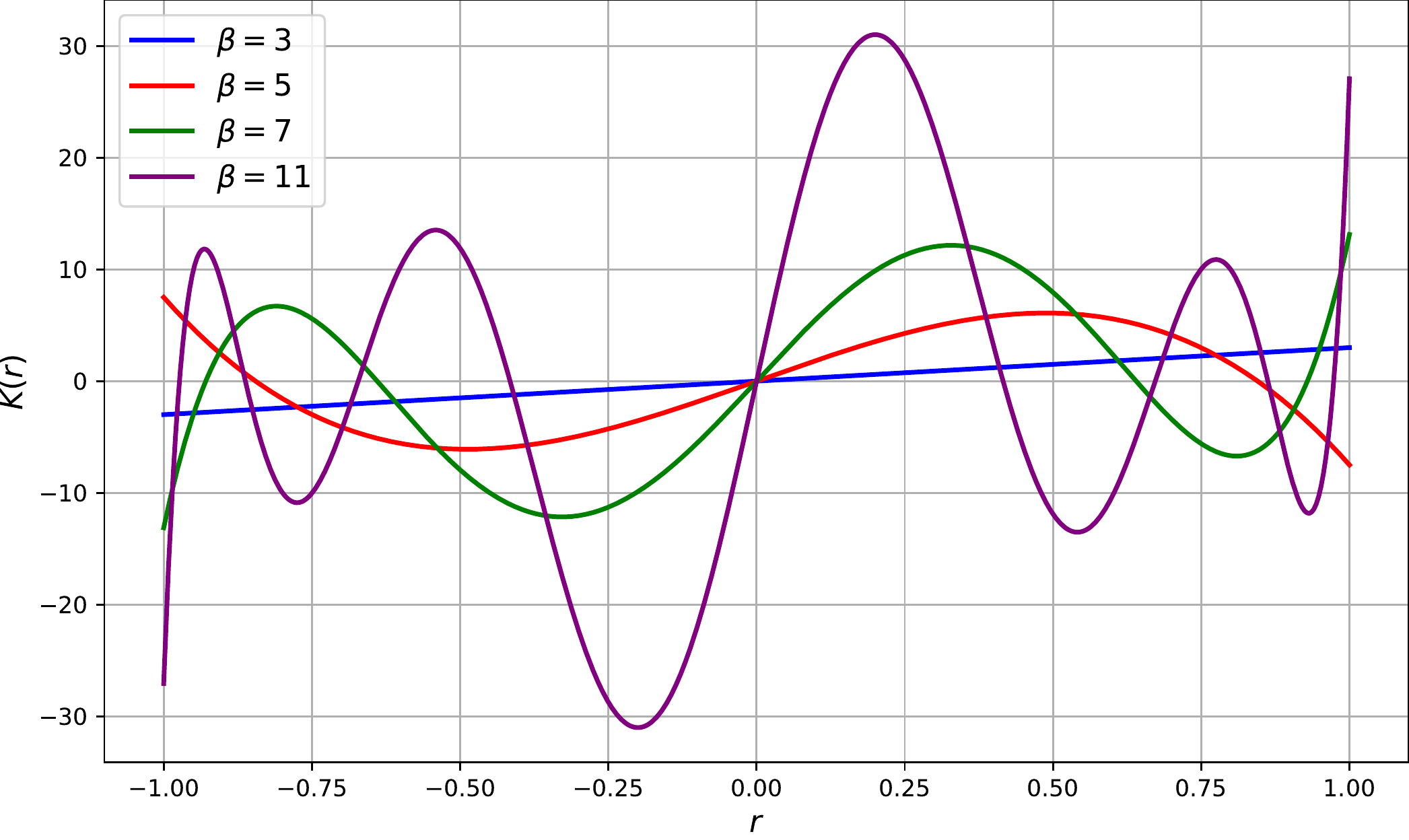}
\caption{Examples of kernels from \eqref{kernels}}
\label{fig:kernels}
\end{figure}

For Theorem \ref{theorem_1} and Theorem \ref{theorem_2}  we need to introduce the constants
\begin{equation}\label{kappa-beta}
    \kappa_{\beta} = \int|u|^{\beta}|K(u)|\,du
\end{equation}
and 
\begin{equation}\label{kappa-squared}
    \kappa =\int K^2(u) \, du.
\end{equation}
It is proved in \cite{bach2016highly} that $\kappa_{\beta}$ and $\kappa$ do not depend on $n$, they depend only on $\beta$: 
\begin{equation}\label{kappa-beta-bound}
    \kappa_{\beta} \leq 2\sqrt{2} (\beta - 1),
\end{equation}
\begin{equation}\label{kappa-squared-bound}
    \kappa \leq \sqrt{3} \beta^{\nicefrac{3}{2}}.
\end{equation}

\section{Theorems}\label{section:theorems} 
In this section we prove upper bounds on the optimization error of Algorithm \ref{algo1} for strongly convex function (Theorem \ref{theorem_1}) and for convex function (Theorem \ref{theorem_2}).

\begin{theorem} \label{theorem_1}
Let $f \in {\cal F}_{\gamma, \beta}(L)$ with $\gamma$, $L > 0$ and $\beta > 2$. 
Let Assumption \ref{noise-ass} hold 
and let $Q$ be a convex compact subset of $\R^n$.
Let $f$ be $G$-Lipschitz on the Euclidean $\tau_1$-neighborhood of $Q$. 

Then the optimization error of averaged estimator $\overline{x}_N = \frac{1}{N}\sum\limits_{k=1}^N x_k$ where the points $x_k$ are given by Algorithm \ref{algo1} with parameters
\begin{equation*}
    \tau_k = \left(\dfrac{3\kappa\sigma^2n}{2(\beta-1)(\kappa_\beta L)^2}\right)^{\frac{1}{2\beta}}k^{-\frac{1}{2\beta}}, \quad \alpha_k=\dfrac{2}{\gamma k}, \quad k =1,\dots, N
\end{equation*}
satisfies
\begin{equation*}
    \mathbb{E}\left[f(\overline{x}_N)-f(x^*)\right]\leq \dfrac{1}{\gamma} \left(n^{2-\frac{1}{\beta}}\dfrac{A_1}{N^{\frac{\beta-1}{\beta}}}+A_2\dfrac{n(1+\ln N)}{N} \right),
\end{equation*}
where $A_1=3\beta(\kappa \sigma^2)^{\frac{\beta-1}{\beta}}(\kappa_{\beta}L)^{\frac{2}{\beta}}$, $A_2=c^*\kappa G^2$, $\kappa_{\beta}$ and $\kappa$ are constants depending only on $\beta$, see \eqref{kappa-beta} and \eqref{kappa-squared}. 
\end{theorem}

\begin{proof}
{\bf Step 1.}
Fix an arbitrary $x \in Q$. As $x_{k+1}$ is the Euclidean projection we have $\|x_{k+1} - x\|^2 \leq \| x_k - \alpha_k \widetilde{g_k} - x \|^2 $ which is equivalent to 
\begin{equation}
    \langle \widetilde{g_k}, x_k-x \rangle \leq \dfrac{\| x_k - x \|^2 - \| x_{k+1} - x \|^2}{2\alpha_k} + \dfrac{\alpha_k}{2} \| \widetilde{g_k} \|^2.
\end{equation}

By the strong convexity assumption we have
\begin{equation}
    f(x_k) - f(x) \leq \langle \nabla f(x_k), x_k - x \rangle - \dfrac{\gamma}{2} \| x_k - x \|^2.
\end{equation}

Combining the last two inequations we obtain
\begin{equation}
\begin{split}
f(x_k) - f(x) \leq & \langle \nabla f(x_k) - \widetilde{g_k}, x_k - x \rangle 
+ \dfrac{\| x_k - x \|^2 - \| x_{k+1} - x \|^2}{2\alpha_k} \\
 & + \dfrac{\alpha_k}{2} \| \widetilde{g_k} \|^2
    - \dfrac{\gamma}{2} \| x_k - x \|^2. \\
\end{split}
\end{equation}

Taking conditional expectation given $x_k$ with respect to $r_{k}$, $\xi_{k}$ and $\xi'_{k}$ we obtain
\begin{equation}\label{regret_bound}
\begin{split}
f(x_k) - f(x) \leq & \langle \nabla f(x_k) -  \E \left[\widetilde{g_k} | x_k \right] , x_k - x \rangle  + \dfrac{\alpha_k}{2} \E \left[ \| \widetilde{g_k} \|^2 | x_k \right] \\
& +  \dfrac{\| x_k - x \|^2 - \E \left[\| x_{k+1} - x \|^2 | x_k \right]}{2\alpha_k} - \dfrac{\gamma}{2} \| x_k - x \|^2. \\
\end{split}
\end{equation}

{\bf Step 2 (Bounding bias term).} Our aim is to bound the first term in \eqref{regret_bound}, namely $\langle \nabla f(x_k) -  \E \left[\widetilde{g_k} | x_k \right] , x_k - x \rangle $.
Using the Taylor expansion we have
\begin{equation}
\begin{split}
    f\left(x_k + \tau_k r_k e_k\right)
    =& f(x_k) + \langle \nabla f(x_k), \tau_k r_k e_k \rangle \\
    +& \sum_{2\leq |m| \leq l} \dfrac{(\tau_k r_k)^{|m|}}{m!} D^{(m)}f(x_k) e_k^{m} + R(\tau_k r_k e_k),
\end{split}
\end{equation}
where by assumption $|R(\tau_k r_k e_k)| \leq L \|\tau_k r_k e_k\|^{\beta} = L (\tau_k \cdot |r_k|)^{\beta} $.
Thus, 
\begin{equation}\label{grad_taylor_expansion}
\begin{split}
    \widetilde{g_k} = &
     \Bigl( \langle \nabla f(x_k), \tau_k r_k e_k \rangle 
    + \sum_{2\leq |m| \leq l, |m| \text{ odd}} \dfrac{(\tau_k r_k)^{|m|}}{m!} D^{(m)}f(x_k) e_k^{m}\\
    &+ \frac{1}{2} R(\tau_k r_k e_k) - \frac{1}{2} R(-\tau_k r_k e_k) + \xi_k - \xi'_k \Bigl) \dfrac{n}{\tau_k} K(r_k) e_k .
\end{split}
\end{equation}

Using the properties of the smoothing kernel $K$, independence of $e_k$ and $r_k$ (Assumption \ref{noise-ass}) and the fact that $\E\left[ e_k e_k^{T} \right] = \frac 1 n \mathbb{I}_{n \times n}$ we obtain
\begin{equation}\label{first_order_term}
    \E_{e_k, r_k} \left[ \left\langle \nabla f(x_k), \tau_k r_k e_k \right\rangle \dfrac{n}{\tau_k} K(r_k) e_k \big| x_k \right] = \nabla f(x_k).
\end{equation}

Using the fact that $\E \left[ r_k^{|m|} K(r_k) \right] = 0$ if $2 \leq |m| \leq l$ or $|m| = 0$ and Assumption \ref{noise-ass} we have
\begin{equation}\label{mid_order_term}
    \Bigl(\sum_{2\leq |m| \leq l, |m| \text{ odd}} \dfrac{(\tau_k r_k)^{|m|}}{m!} D^{(m)}f(x_k) e_k^{m}
    + \xi_k - \xi'_k \Bigl) \dfrac{n}{\tau_k} K(r_k) e_k = 0.
\end{equation}

Combining \eqref{grad_taylor_expansion}, \eqref{first_order_term} and \eqref{mid_order_term} and using the definition of $\kappa_{\beta}$ we obtain
\begin{multline}\label{step2_prefinal}
    \left|   \langle \nabla f(x_k) -  \E \left[\widetilde{g_k} | x_k \right] , x_k - x \rangle \right|
     = \\
     =  \left| \E\left[ \left( \frac{1}{2} R(\tau_k r_k e_k) - \frac{1}{2} R(-\tau_k r_k e_k) \right) \dfrac{n}{\tau_k} K(r_k) \la e_k, x_k - x \ra \middle| x_k \right] \right| \\
     \leq L\tau_k^{\beta - 1} \cdot \E_{r_k} \left[ |r_k|^{\beta} K(r_k) \right] \cdot n\left|\E_{e_k} \left[\la e_k, x_k - x \ra \middle| x_k \right] \right| \\ 
     \leq \kappa_{\beta} L \sqrt{n} \tau_k^{\beta - 1} \| x_k - x\|, \\
\end{multline}
where in the last inequality the fact that $\left|\E_{e} \left[\la e, s \ra \right] \right|^2 \leq \E_{e} \left[\la e, s \ra^2 \right] = \frac{\|s\|^2}{n} $ was used (the fact from concentration measure theory). Applying the inequality $ab \leq \nicefrac{1}{2} (a^2 + b^2)$ to the last expression in \eqref{step2_prefinal}  we finally get 
\begin{equation}\label{step2}
    \left|   \langle \nabla f(x_k) -  \E \left[\widetilde{g_k} | x_k \right] , x_k - x \rangle \right|
    \leq  \dfrac{(\kappa_{\beta} L)^2}{\gamma} n \tau_k^{2(\beta-1)} + \dfrac{\gamma}{4} \|x_k-x\|^2 .
\end{equation}

{\bf Step 3 (Bounding second moment of gradient estimator). }
Our aim is to estimate $\E \left[ \| \widetilde{g_k} \|^2 | x_k \right]$ which is the second term in \eqref{regret_bound}. The expectation here is with respect to $r_k$, $\xi_k$ and $\xi'_k$. To lighten the presentation and withous loss of generality we drop the lower script $k$ in all quantities. 

We have
\begin{equation}
\begin{split}
    \| \widetilde{g} \|^2 =& \dfrac{n^2}{4\tau^2} \|(f(x+\tau r e) - f(x - \tau r e) + \xi - \xi') K(r) e\|^2 \\
    =& \dfrac{n^2}{4\tau^2} \left((f(x+\tau r e) - f(x - \tau r e) + \xi - \xi') \right)^2 K^2(r).\\
\end{split}
\end{equation}

Using the inequality $(a+b+c)^2 \leq 3(a^2 + b^2 + c^2)$ and Assumption \ref{noise-ass} we get
\begin{equation}\label{dispersion_bound}
    \E \left[ \| \widetilde{g} \|^2 | x \right] 
    \leq \dfrac{3n^2}{4\tau^2} \left( \E\left[(f(x+\tau r e) - f(x - \tau r e))^2 K^2(r) \middle| x \right] + 2\kappa \sigma^2 \right).
\end{equation}

Lemma 9 in \cite{shamir2017optimal} states that for any function $f$ which is $G-Lipschitz$ with respect to 2-norm, it holds that if $e$ is uniformly distributed on the Euclidean unit sphere, then
\begin{equation}\label{lemma_shamir}
    \sqrt{\E \left[ (f(e) - \E[f(e)])^4 \right]} \leq  \dfrac{cG^2}{n},
\end{equation}
where $c < 3$ is a positive numerical constant.

Using \eqref{lemma_shamir}, symmetry of Euclidean unit sphere and the inequality $(a+b)^2 \ \leq 2(a^2 + b^2)$ we obtain
\begin{multline}
    \E \left[ \left(f(x+e) -f(x-e)\right)^2 \middle| x \right] = \E_e \left[ \left(f(x+e) -f(x-e)\right)^2  \right] \\
    \leq \E_e \left[ \left( \left(f(x+e) - \E_e [f(x+e)]\right) - \left(f(x-e) - \E_e [f(x-e)]\right) \right)^2  \right]\\
    \leq 2\E_e \left[ \left(f(x+e) - \E_e [f(x+e)]\right)^2 \right] + 2\E_e \left[ \left(f(x-e) - \E_e [f(x-e)]\right)^2 \right] \\
    \leq 2\sqrt{\E_e \left[ \left(f(x+e) - \E_e [f(x+e)]\right)^4 \right]} + 2\sqrt{\E_e \left[ \left(f(x-e) - \E_e [f(x-e)]\right)^4 \right]} \\
    \leq \dfrac{4cG^2}{n},
\end{multline}
so we have
\begin{equation}\label{delta_f_bound}
    \E \left[ \left(f(x+\tau re) -f(x-\tau re)\right)^2 \middle| x \right] \leq \dfrac{4c(\tau r)^2 G^2}{n} \leq \dfrac{4c \tau^2 G^2}{n}.
\end{equation}

By substituting \eqref{delta_f_bound} into \eqref{dispersion_bound}, using independence of $e$ and $r$ and returning the lower script $k$ we finally get
\begin{equation}\label{step3}
    \E \left[ \| \widetilde{g_k} \|^2 | x \right] 
    \leq \kappa \left(c^* n G^2  + \dfrac{3(n\sigma)^2}{2\tau_k^2} \right),
\end{equation}
where $c^* = 3c$.

{\bf Step 4. } Let $\rho_k^2$ denote $\E[\|x_k-x\|^2]$. Substituting \eqref{step2} and \eqref{step3} into \eqref{regret_bound}, taking full expectation and summing over $k$ we obtain \begin{equation}\label{step4_begin}
\begin{split}
    \sum_{k=1}^N \E[f(x_k) - f(x)] 
    \leq& \sum_{k=1}^N \left(\dfrac{(\kappa_{\beta} L)^2}{\gamma} n \tau_k^{2(\beta-1)} 
    + \dfrac{\alpha_k}{2} \kappa \left(c^* n G^2  + \dfrac{3(n\sigma)^2}{2\tau_k^2} \right) \right)\\
    & +\sum_{k=1}^N \left( \dfrac{\rho_k^2 - \rho_{k+1}^2}{2\alpha_k} - \left(\dfrac{\gamma}{2} - \dfrac{\gamma}{4} \right) \rho_k^2 \right).\\
\end{split}
\end{equation}

Let $\rho_{N+1}^2 = 0$. Then setting $\alpha_k = \dfrac{2}{\gamma k}$ yields 
\begin{equation}\label{sum_rho}
\begin{split}
    \sum_{k=1}^{N} \left( \dfrac{\rho_k^2 - \rho_{k+1}^2}{2\alpha_k} - \dfrac{\gamma}{4} \rho_k^2 \right) &\leq
    \rho_1^2 \left( \dfrac{1}{2\alpha_1} - \dfrac{\gamma}{4} \right) + \sum_{k=2}^{N+1} \rho_k^2 \left( \dfrac{1}{2\alpha_k} - \dfrac {1}{2\alpha_{k-1}}- \dfrac{\gamma}{4}\right) \\
    &=\rho_1^2 \left( \dfrac{\gamma}{4}-\dfrac{\gamma}{4} \right) + \sum_{k=2}^{N+1} \rho_k^2 \left( \dfrac{\gamma}{4}-\dfrac{\gamma}{4} \right) = 0.
\end{split}
\end{equation}

Substituting \eqref{sum_rho} into \eqref{step4_begin} with $\alpha_k = \frac{2}{\gamma k}$ we obtain 
\begin{equation}\label{step4_prefinal}
\begin{split}
    \sum_{k=1}^N \E[f(x_k) - f(x)]  
    &\leq \dfrac{1}{\gamma} \sum_{k=1}^N   \left((\kappa_{\beta} L)^2 n \tau_k^{2(\beta-1)} 
    + \kappa \left(c^* n G^2  + \dfrac{3(n\sigma)^2}{2\tau_k^2} \right) \dfrac{1}{k} \right)\\
    &= \dfrac{1}{\gamma} \sum_{k=1}^N \left( \left[n \cdot (\kappa_{\beta} L)^2 \tau_k^{2(\beta-1)} +
     n^2 \cdot \dfrac{3\kappa \sigma^2}{2k\tau_k^2} \right] +  \dfrac{c^* \kappa n G^2}{k}\right).\\
\end{split}
\end{equation}

If $\sigma > 0$ then $\tau_k = {\left(\dfrac{3\kappa \sigma^2 n}{2(\beta -1)(\kappa_{\beta} L)^2}\right)}^{\frac{1}{2\beta}} k^{-\frac{1}{2\beta}}$ is the minimizer of square brackets. Plugging this $\tau_k$ in \eqref{step4_prefinal} and using two inequalities: for the expression in square brackets $\sum\limits_{k=1}^N k^{-1+\nicefrac{1}{\beta}} \leq \beta N^{\nicefrac{1}{\beta}}$ (if $\beta > 2$) and for the term after square brackets $\sum\limits_{k=1}^{N} \frac{1}{k} \leq 1 + \ln N$ we get
\begin{equation}
    \sum_{k=1}^N \E[f(x_k) - f(x)] 
    \leq \dfrac{1}{\gamma} \left(n^{2-\frac{1}{\beta}} A_1 N^{\frac{1}{\beta}}+A_2n(1+\ln{N}) \right)
\end{equation}
with $A_1$ and $A_2$ from the formulation of Theorem \ref{theorem_1}. Due to the convexity of $f$ we finally prove the theorem
\begin{equation}
    \mathbb{E}\left[f(\overline{x}_N)-f(x^*)\right]\leq \dfrac{1}{\gamma} \left(n^{2-\frac{1}{\beta}}\dfrac{A_1}{N^{\frac{\beta-1}{\beta}}}+A_2\dfrac{n(1+\ln{N})}{N} \right).
\end{equation}
\EndProof
\end{proof}

We emphasize that the usage of kernel smoothing technique, measure concentration inequalities and the assumption that $\xi_k$ is independent from $e_k$ or $r_k$ (Assumption \ref{noise-ass}) lead to the results better  than the state-of-the-art ones for $\beta > 2$ (see Table \ref{tab-1} and Table \ref{tab-2}). The last assumption also allows us not to assume neither zero-mean of $\xi_k$ and $\xi'_k$ nor i.i.d of $\{\xi_k\}_{k=1}^{N}$ and $\{\xi'_k\}_{k=1}^{N}$.

\begin{theorem} \label{theorem_2}

Let $f \in {\cal F}_{\beta}(L)$ with $\gamma$, $L > 0$ and $\beta > 2$. 
Let Assumption \ref{noise-ass} hold 
and let $Q$ be a convex compact subset of $\R^n$.
Let $f$ be $G$-Lipschitz on the Euclidean $\tau_1$-neighborhood of $Q$. Let $\overline{x}_N$ denote $\frac{1}{N}\sum\limits_{k=1}^N x_k$.

Then we achieve the optimization error
$\mathbb{E}\left[f(\overline{x}_N)-f(x^*)\right]\leq \varepsilon$
after $N(\varepsilon)$ steps of Algorithm \ref{algo1} with settings from Theorem \ref{theorem_1} for the regularized function: $f_{\gamma}(x):=f(x)+\frac{\gamma}{2} \| x-x_0 \|^2$, where $\gamma \leq \frac{\varepsilon}{R^2}$, $R=\|x_0-x^*\|$, $x_0 \in Q$ -- arbitrary point.

\begin{equation*}
    N(\varepsilon)=\max\left\{ \left(R\sqrt{2A_1}\right)^{\frac{2\beta}{\beta-1}}\dfrac{n^{2+\frac{1}{\beta-1}}}{\varepsilon^{2+\frac{2}{\beta-1}}},\left(R\sqrt{2c'A_2}\right)^{2(1+\rho)}\dfrac{n^{1+\rho}}{\varepsilon^{2(1+\rho)}}\right\},
\end{equation*}
where  $A_1=3\beta(\kappa \sigma^2)^{\frac{\beta-1}{\beta}}(\kappa_{\beta}L)^{\frac{2}{\beta}}$, $A_2=c^*\kappa G^2$ -- constants from Theorem \ref{theorem_1}, $\rho > 0$ -- arbitrarily small positive number. 

\end{theorem}

\begin{proof}
{\bf Step 1.} Let $x^*$ and $x_{\gamma}^*$ denote $\arg\min\limits_{x\in Q} f(x)$ and $\arg\min\limits_{x\in Q} f_{\gamma}(x)$ respectively. 
Setting $\gamma = \frac{\varepsilon}{R^2}$ and using the inequality $f_{\gamma}(x_{\gamma}^*) \leq f_{\gamma}(x^*)$ we obtain
\begin{equation}\label{connection}
\begin{split}
    f(\overline{x}_N) - f(x^*) 
    &= f_{\gamma}(\overline{x}_N) - f_{\gamma}(x^*) 
    -\dfrac{\gamma}{2} \|\overline{x}_N-x_0\|^2 
    +\dfrac{\gamma}{2} \|x^* - x_0\|^2\\
    &\leq f_{\gamma}(\overline{x}_N) - f_{\gamma}(x^*) 
    +\dfrac{\gamma}{2} \|x^* - x_0\|^2 \\
    &\leq f_{\gamma}(\overline{x}_N) - f_{\gamma}(x_{\gamma}^*) 
    +\dfrac{\varepsilon}{2}.\\
\end{split}
\end{equation}

{\bf Step 2.} 
Now we apply Theorem \ref{theorem_1} for $f_{\gamma}(x)$ and bound RHS by $\frac{\varepsilon}{2}$:
\begin{equation}\label{ref-theorem-1}
    \mathbb{E}\left[f_{\gamma}(\overline{x}_N)-f_{\gamma}(x^*)\right]\leq \dfrac{1}{\gamma} \left(n^{2-\frac{1}{\beta}}\dfrac{A_1}{N^{\frac{\beta-1}{\beta}}}+A_2\dfrac{n(1+\ln{N})}{N} \right) \leq \dfrac{\varepsilon}{2}.
\end{equation}

The inequality \eqref{ref-theorem-1} is done if ($\gamma=\frac{\varepsilon}{R^2}$)
\begin{equation}\label{th2-step-2}
    \max\left\{n^{2-\frac{1}{\beta}}\dfrac{A_1}{N^{\frac{\beta-1}{\beta}}}, A_2\dfrac{n(1+\ln{N})}{N} \right\} 
    \leq \dfrac{\gamma \varepsilon}{2}
    = \dfrac{\varepsilon^2}{2R^2} .
\end{equation}

It is true that  $1 + \ln N \leq c' N^{\frac{\rho}{\rho + 1}}$ for some $c' > 0$. So the inequality \eqref{th2-step-2} holds if
\begin{equation}
    N \geq \max\left\{ \left(R\sqrt{2A_1}\right)^{\frac{2\beta}{\beta-1}}\dfrac{n^{2+\frac{1}{\beta-1}}}{\varepsilon^{2+\frac{2}{\beta-1}}},\left(R\sqrt{2c'A_2}\right)^{2(1+\rho)}\dfrac{n^{1+\rho}}{\varepsilon^{2(1+\rho)}}\right\}.
\end{equation}

The inequalities \eqref{connection} and \eqref{ref-theorem-1}  yield $\mathbb{E}\left[f(\overline{x}_N)-f(x^*)\right]\leq \varepsilon$.
\EndProof
\end{proof}

\section{Numerical Experiments}\label{section:experiments}
In our experiment \cite{novitskii} we compare the Algorithm \ref{algo1} (with $\beta = 3$ and $\beta = 5$) proposed in this paper with Gasnikov's one-point method and with Akhavan's method for the special case $\beta = 2$. 

We consider the problem of the minimization of the quadratic function
\begin{equation*}
    f(x) = \dfrac{1}{4} x_1^2 + x_2^2 + 4x_3^2
\end{equation*}
on the Euclidean ball $Q = \{ x \in \R^3: \, \|x\|\leq 1\}$.

The starting point is $x_0$ with $\|x_0\| = \nicefrac{1}{2}$.  The dependency of $f(\overline{x}_N)-f(x^*)$ (optimization error) on $N$ (iteration number) is presented on the Figure \ref{fig:experiment}. The optimization error has its mean and 0.95-confidence interval.  As the Lipshitz constants for the quadratic oracle are equal to zero, for the Algorithm \ref{algo1} we choose $L = 0.01$. 

\begin{figure}[h]
\centering
\includegraphics[width =  0.78\linewidth]{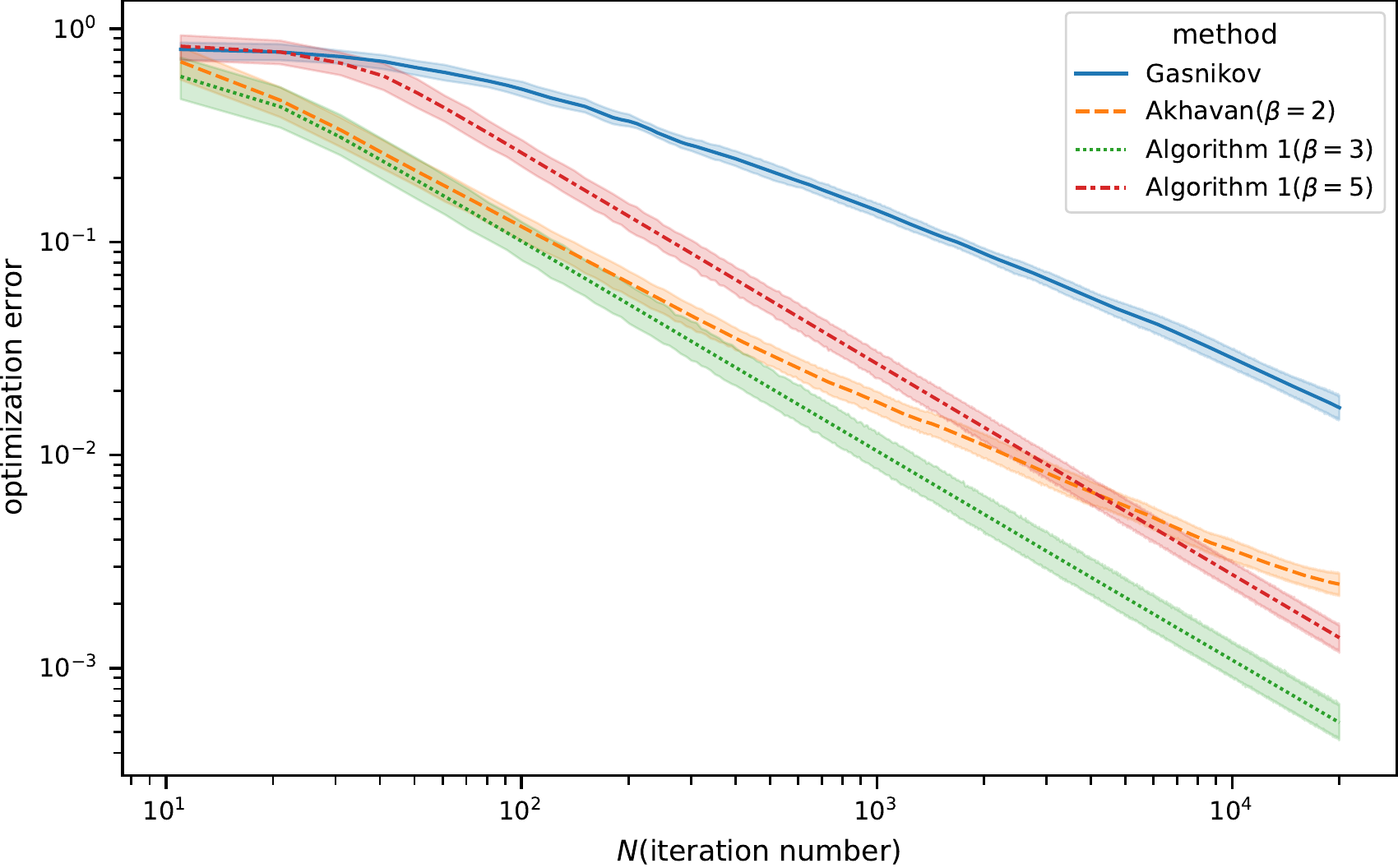}
\caption{Dependency of optimization error of Algorithm \ref{algo1} on iteration}
\label{fig:experiment}
\end{figure}

We see on the Figure \ref{fig:experiment} that the usage of higher-order smoothness by Algorithm \ref{algo1} helps to overcome the methods which do not use this. 

\section{Discussion and related work}\label{section:discussion}
The results of this paper can be generalized  for the saddle-point problems. Recently GANs and Reinforcement Learning caused a big interest for saddle-point problems, see \cite{sadiev2020zeroth}. 

Another possible genelization of this paper is obtaining the large probability bounds for optimization error. We cannot obtain upper bounds in terms of 
of large deviation probability (not in terms of expectation) under the Assumption \ref{noise-ass}. 
The exploiting of higher order smoothness with the help of kernels under rather general noise assumptions (non-zero mean) causes big variation $\|\widetilde{g_k} - \nabla f(x_k)\|$ and this can causes the problems with large deviation probability rates. 

It remains an open question whether large deviation probability can be obtained under non-zero mean noise. And also it remains an open question whether better dependence of optimization error on the dimenstion $n$ and strong convexity parameter $\gamma$ can be obtained.

{\bf Acknowledgements.} We would like to thank Alexandre B. Tsybakov for helpful remarks about Tables 1 and 2.

\bibliographystyle{spmpsci}
\bibliography{literature}

\end{document}